\documentclass{jcmlatex}
\def\JCMvol{35}
\def\JCMno{4}
\def\JCMyear{2017}
\def\JCMreceived{November 12, 2015}
\def\JCMrevised{June 2, 2016}
\def\JCMaccepted{June 27, 2016}
\def\DOI{10.4208/jcm.1606-m2015-0452}

\usepackage{graphicx}
\usepackage{amsmath}
\usepackage{graphics}
\usepackage{amsfonts}
\usepackage{amssymb}%
\usepackage{subfigure}
\usepackage{color}
\usepackage{url}

\usepackage{graphicx}
\usepackage{epstopdf}
\newtheorem{Thm}{Theorem}
\newtheorem{Lemm}{Lemma}
\newtheorem{Prop}{Property}

\newtheorem{Assump}{Assumption}

\newcommand{\Min}{{\mathop{\mathrm{minimize}}}}
\newcommand{\diagnal}{{\mathbf{diag}}}
\newcommand{\nullspc}{{\mathbf{null}}}

\begin{document}

\markboth{J.S. ~ZENG AND W.T.~YIN} {ExtraPush for Convex Smooth Decentralized Optimization}

\title{EXTRAPUSH FOR CONVEX SMOOTH DECENTRALIZED OPTIMIZATION OVER DIRECTED NETWORKS}

\author{Jinshan Zeng
\thanks{College of Computer Information Engineering, Jiangxi Normal University, Nanchang, \\Jiangxi 330022, China \\ Email: jinshanzeng@jxnu.edu.cn}
\and
 Wotao Yin
\thanks{Department of Mathematics, University of California, Los Angeles, CA 90095, USA\\ Email: wotaoyin@math.ucla.edu}}

\maketitle

\begin{abstract}
In this note, we extend the algorithms Extra \cite{Yin-EXTRA2015} and subgradient-push \cite{Nedic-SubgradientPush2015} to a new algorithm {\em ExtraPush} for  consensus optimization with convex differentiable objective functions over a \emph{directed} network. {When the stationary distribution of the network can be computed in advance}, we propose a simplified algorithm called \emph{Normalized ExtraPush}. {Just like Extra, both ExtraPush and Normalized ExtraPush can iterate with a fixed step size.  But unlike Extra, they can take a column-stochastic mixing matrix, which is not necessarily doubly stochastic. Therefore, they remove the undirected-network restriction of Extra. Subgradient-push, while also works for \emph{directed} networks, is slower on the same type of problem  because it must use a sequence of diminishing step sizes.}

We present preliminary analysis for ExtraPush under a bounded sequence assumption. For Normalized ExtraPush, we show that it naturally produces a bounded, linearly convergent sequence provided that the objective function is strongly convex.

{In our numerical experiments, ExtraPush and Normalized ExtraPush performed similarly well. They are significantly faster than subgradient-push, even when we hand-optimize the step sizes for the latter.}
\end{abstract}



\section{Introduction}
\label{sec:into}

We consider the following consensus optimization problem defined on a directed, strongly connected network of $n$ agents:
\begin{equation}
\label{Eq:multi-agentOPT}
\mathop{\Min}_{x\in \mathbf{R}^p} f(x) \triangleq \sum_{i=1}^n f_i(x),
\end{equation}
where $f_i$ is a proper, closed, convex, differentiable function only known to the agent $i$.

The model (\ref{Eq:multi-agentOPT}) finds applications in decentralized averaging, learning, estimation, and control.
For a stationary network with \emph{bi-directional} communication, the existing algorithms include the (sub)gradient methods \cite{Chan-FastGradient2012,Jakovetic-FastGradient2014,Matei-Subgradient2011,Nedic-Subgradient2009,Yin-EXTRA2015,Yin-gradient2014},
and the primal-dual domain methods such as the decentralized alternating direction method of multipliers (DADMM) \cite{Schizas-DADMM2008,Shi-DADMM2014}.

This note focuses on a \emph{directed} network (with \emph{directional} communication), where the research of decentralized optimization is pioneered by the works {\cite{Tsianos2012a-Alg,Tsianos2012b-Theory,Tsianos2013-PhdThsis}}.
When communication is bi-directional, algorithms can use a symmetric and doubly-stochastic mixing matrix to obtain a consensual solution; however, once the communication is directional, the mixing matrix becomes generally asymmetric and only column-stochastic. {Also consider the setting where  each agent broadcasts its information to its neighbors, yet an agent may not receive the information from a neighbor. An agent can weigh its  information (both from itself and received from its neighbors) so that the total weights add up to 1, but an agent cannot ensure that its broadcasted information receives weights that precisely add up to exactly 1. Therefore, only each column of the mixing matrix sums to 1.}
 In the column-stochastic setting, the push-sum protocol \cite{Kempe-Push-sumw2003} can be used to obtain a stationary distribution for the mixing matrix.

In the symmetric and  doubly-stochastic setting, if the objective is Lipschitz-differentiable, the gradient-based algorithm Extra \cite{Yin-EXTRA2015}  converges at the rate of $O(1/t)$, where $t$ is the iteration number. In the column-stochastic setting, the best rate is $O(\ln t /\sqrt{t})$ from the subgradient-based algorithm \cite{Nedic-SubgradientPush2015}. {We address the open question of how to} take advantage of the gradient of a Lipschitz-differentiable objective. We make an attempt in this note to combine ideas in \cite{Nedic-SubgradientPush2015,Yin-EXTRA2015} and present our preliminary  results.

Specifically, we propose \emph{ExtraPush}, which is a two-step iteration like Extra and incorporates the push-sum protocol. At each iteration, the Extra variables are approximately normalized by the current push-sum variables.
{When the stationary distribution of the network can be easily computed,} we propose to first apply the push-sum protocol to obtain the stationary distribution and then run the two-step iteration Normalized ExtraPush. At each iteration, its running variables are normalized by the stationary distribution.

Our algorithms are essentially the same as found in the recent work by Xi and Khan {\cite{Xi-Khan2015}}. They attempted to prove convergence for a strongly convex objective function. They noticed that a certain matrix that is important to the analysis (as a part of their convergence metric) is positive semi-definite. Our analysis also uses this property. However, their analysis breaks down due to incorrect assumptions. {More specifically, each function $f_i$ is assumed in {\cite{Xi-Khan2015}} to be strongly convex and also has a bounded and Lipschitz gradient (i.e., its gradient is bounded and Lipschitz continuous). However, no function can satisfy these assumptions simultaneously since gradients of a strongly convex are strictly increasing and unbounded.}

It is  worth noting that our algorithm can be  applied to a time-varying directed network after a straightforward modification; our convergence proof, however, will need a significant change.

The rest of this note is organized as follows. Section 2 introduces the problem setup and preliminaries. Section 3 develops ExtraPush and Normalized ExtraPush.
Section 4 establishes the optimality conditions for ExtraPush and shows its convergence under the boundedness assumption.
Section 5 assumes that the objective is strongly convex and shows that Normalized ExtraPush produces a bounded sequence that converges linearly.
{Section 6 presents our numerical simulation results.
We conclude this paper in Section 7.}

\textbf{Notation:} Let ${\bf I}_n$ denote an identity matrix with the size $n\times n$, and ${\bf 1}_{n\times p} \in \mathbf{R}^{n\times p}$ denote the \emph{matrix} with all entries equal to 1. We also use ${\bf 1}_{n} \in \mathbf{R}^{n}$ as a vector of all $1$'s.
For any \emph{vector} $x$, we let $x_i$ denote its $i$th component and $\diagnal(x)$ denote the diagonal matrix generated by $x$.
For any matrix $X$, $X^T$ denotes its transpose, $X_{ij}$ denotes its $(i,j)$th component, and $\|X\| \triangleq \sqrt{\langle X, X \rangle}=\sqrt{\sum_{i,j}X_{ij}^2}$ denotes its Frobenius norm.
The largest and smallest eigenvalues of matrix $X$ are denoted as $\lambda_{\max}(X)$ and $\lambda_{\min}(X)$, respectively.
For any matrix $B\in \mathbf{R}^{m\times n}$, $\nullspc(B) \triangleq \{x\in \mathbf{R}^n| Bx=0\}$ is the null space of $B$.
{Given a matrix $B\in \mathbf{R}^{m\times n}$, by $Z\in \nullspc(B)$, we mean that
each column of $Z$ lies in $\nullspc(B)$.}
The smallest \textit{nonzero} eigenvalue of a symmetric positive semidefinite matrix $X\neq {\bf 0}$ is denoted as $\tilde{\lambda}_{\min}(X)$, which is strictly positive.
For any positive semidefinite matrix $G\in\mathbf{R}^{n\times n}$ (not necessarily symmetric in this paper), we use the notion $\|X\|_G^2 \triangleq \langle X, GX \rangle$ for a matrix $X \in \mathbf{R}^{n\times p}$.

\section{Problem Reformulation}
\label{sec:manu}

\subsection{Network}

Consider a \emph{directed} network ${\cal G} = \{V,E\}$, where $V$ is the vertex set and $E$ is the edge set.
Any edge $(i,j)\in E$ represents a directed arc from node $i$ to node $j$.
The sets of in-neighbors and out-neighbors of  node $i$ are
\begin{align*}
{\cal{N}}_i^{\mathrm{in}} \triangleq \{j: (j,i) \in E\} \cup \{i\},\quad
 {\cal{N}}_i^{\mathrm{out}} \triangleq \{j: (i,j) \in E\} \cup \{i\},
\end{align*}
respectively.
Let $d_i \triangleq |{\cal{N}}_i^{\mathrm{out}}|$ be the out-degree of node $i$.
In ${\cal G}$, {each node $i$ can only send information to its out-neighbors, \emph{not} vice versa}.

To illustrate a mixing matrix for a directed network, consider $A\in \mathbf{R}^{n\times n}$ where 
\begin{equation}
\label{Eq:MixingMatrixA}
\left\{
\begin{array}{ll}
A_{ij}>0, & \text{if}\ j\in {\cal N}_i^{\mathrm{in}}\\
A_{ij}=0, & \text{otherwise.}
\end{array}%
\right.
\end{equation}%
The entries $A_{ij}$ satisfy that, for each node $j$,  $\sum_{i\in V}A_{ij}=1$.
An example is the following mixing matrix
\begin{equation}
\label{Eq:MixingMatrixA*}
A_{ij} =
\left\{
\begin{array}{ll}
1/d_j, & \text{if}\ j\in {\cal N}_i^{in}\\
0, & \text{otherwise,}
\end{array}%
\right.
\end{equation}%
$i,j=1,\ldots, n$, which is used in the subgradient-push method {\cite{Nedic-SubgradientPush2015}}.
See Fig. {\ref{Fig:Network}} for a directed graph ${\cal G}$  and an example of its mixing matrix $A$.
The matrix $A$ is column stochastic and  asymmetric in general.

\begin{figure}[!thb]
\centering
\begin{minipage}[t][][c]{.35\textwidth}
\centering
\includegraphics[scale=0.52]{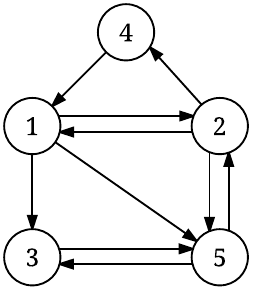}
\end{minipage}
\hspace{10pt}
\begin{minipage}[t][][b]{.35\textwidth}
$\displaystyle
A = \left(
\begin{array}{ccccc}
\frac{1}{4} &\frac{1}{4} & 0 & \frac{1}{2} & 0\\[2pt]
\frac{1}{4} &\frac{1}{4} & 0 & 0 & \frac{1}{3}\\[2pt]
\frac{1}{4}  &0 & \frac{1}{2} & 0 & \frac{1}{3} \\[2pt]
0 &\frac{1}{4} & 0 & \frac{1}{2} & 0\\[2pt]
\frac{1}{4} & \frac{1}{4} & \frac{1}{2} & 0 & \frac{1}{3}\\
\end{array}
\right)$
\end{minipage}\vskip 2mm
\caption{ A directed graph ${\cal G}$ (left) and its mixing matrix $A$ (right).}\vskip -4mm
\label{Fig:Network}
\end{figure}

\begin{Assump}
\label{Assump:Network}
The graph ${\cal G}$ is strongly connected.
\end{Assump}

\begin{Prop}
\label{Prop:A}
Under Assumption {\rm{\ref{Assump:Network}}}, the followings hold (parts (i) and (iv) are results in \cite[Corollary 2]{Nedic-SubgradientPush2015}):
\begin{enumerate}
\item[(i)]
Let $A^t = \overbrace{A\times A \cdots A}^t$ for any $t\in \mathbf{N}$. Then
\begin{align}
\label{Eq:LimofAt}
A^t \rightarrow \phi {\bf 1}_n^T \ \ \mathrm{geometrically\  fast\  as }\ \ t \rightarrow \infty,
\end{align}
for some \emph{stationary distribution} vector $\phi$, i.e., $\phi_i\ge 0$ and $\sum_i^n \phi_i =1$.

\item[(ii)]
$\nullspc({\bf I}_n - \phi {\bf 1}_n^T) = \nullspc({\bf I}_n - A).$

\item[(iii)]
$A\phi = \phi.$

\item[(iv)]
The quantity $\xi \triangleq \inf_t \min_{1\leq i \leq n} (A^t{\bf 1}_n)_i \geq \frac{1}{n^n}>0.$
\end{enumerate}
\end{Prop}

\begin{proof}
 Part (iii) is obvious from (ii) since $\phi \in \nullspc({\bf I}_n - \phi {\bf 1}_n^T)$ and $\sum_i^n \phi_i =1$.
Next, we show part (ii). First, let $z \in \nullspc({\bf I}_n - \phi {\bf 1}_n^T)$, which means $z=\phi {\bf 1}_n^T z$ and thus $Az = A \phi {\bf 1}_n^T z$.
By {\eqref{Eq:LimofAt}}, it is obvious that $A \phi {\bf 1}_n^T = \phi {\bf 1}_n^T$.
Therefore, $Az = \phi {\bf 1}_n^T z = z$ and hence $\nullspc({\bf I}_n - \phi {\bf 1}_n^T) \subseteq \nullspc({\bf I}_n - A).$
On the other hand,  any $z \in \nullspc({\bf I}_n - A),$ equivalently, $z=Az$, obeys $z=A^t z$ for any $t \geq 1$.
Letting $t\to\infty$, it holds that $z = \phi {\bf 1}_n^T z$, that is, $z \in \nullspc({\bf I}_n - \phi {\bf 1}_n^T).$ Therefore, part (ii) holds.\hfill$\Box$
\end{proof}

\subsection{Problem Given in the Matrix Notation}
Let $x_{(i)} \in \mathbf{R}^p$ denote the \emph{local copy} of $x$ at node $i$, and
$x_{(i)}^t$ denote its value at the $t$th iteration.
Throughout the note, we  use the following equivalent form of the problem (\ref{Eq:multi-agentOPT}) using {local copies} of the  variable $x$:
\begin{equation}\label{Eq:consensusProblem}
\Min_{\mathbf{x}}\  {\bf 1}_n^T {\bf f(x)} \triangleq \sum_{i=1}^n f_i(x_{(i)}),\quad  \mathrm{subject\  to}\ x_{(i)} = x_{(j)}, \ \forall i,j\in E,
\end{equation}
where ${\bf 1}_n \in \mathbf{R}^n$ denotes the vector with all its entries equal to 1,
$$
{\bf x} \triangleq \left(
\begin{array}{ccc}
\mbox{---} &x^T_{(1)} & \mbox{---}\\
\mbox{---} &x^T_{(2)} & \mbox{---}\\
  &\vdots &  \\
\mbox{---} &x^T_{(n)} & \mbox{---}\\
\end{array}
\right)
\in \mathbf{R}^{n\times p},\quad
{\bf f(x)} \triangleq \left(
\begin{array}{c}
f_1(x_{(1)})\\
f_2(x_{(2)})\\
\vdots   \\
f_n(x_{(n)})\\
\end{array}
\right)
\in \mathbf{R}^{n}.
$$
In addition, the gradient of ${\bf f}(\bf x)$ is
$$
{\bf \nabla f(x)} \triangleq \left(
\begin{array}{ccc}
\mbox{---} &\nabla f_1(x_{(1)})^T & \mbox{---}\\
\mbox{---} &\nabla f_2(x_{(2)})^T & \mbox{---}\\
  &\vdots &  \\
\mbox{---} &\nabla f_n(x_{(n)})^T & \mbox{---}\\
\end{array}
\right)
\in \mathbf{R}^{n\times p}.
$$
The $i$th rows of the above matrices $\mathbf{x}$ and $\nabla \mathbf{f}(\mathbf{x})$,  and vector
${\bf f(x)}$, correspond to agent $i$. For simplicity, one can treat $p=1$ throughout this paper.

For a  vector $\bar{x} \in \mathbf{R}^n$, let $\bar{x}^{\mathrm{ave}} \triangleq \frac{1}{n}(\sum_{i=1}^n \bar{x}_i)\in\mathbf{R}$.
A special case of (\ref{Eq:consensusProblem}) is the well-known average consensus problem, where
$f_i(x_{(i)}) = \frac{1}{2}(x_{(i)} - \bar{x}_{i})^2$
for each node $i$ and the solution is $x_{(i)}= \bar{x}^{\mathrm{ave}}$ for all $i$.

\section{Proposed Algorithms}

\subsection{Reviews of Extra and subgradient-push}
Extra {\cite{Yin-EXTRA2015}} is a ``two-step'' iterative algorithm for solving (\ref{Eq:consensusProblem}) over an undirected network.
Let $W \in \mathbf{R}^{n\times n}$ be a symmetric and doubly stochastic mixing matrix, and $\bar{W} \triangleq \frac{{\bf I}_n+W}{2}$.
The Extra iteration is 
\begin{equation}
\label{Eq:EXTRA}
{\bf x}^{t+2} = ({\bf I}_n + W){\bf x}^{t+1} - \bar{W} {\bf x}^t -\alpha (\nabla {\bf f}({\bf x}^{t+1}) - \nabla {\bf f}({\bf x}^{t})),\;\;\;  t=0,1,\cdots,
\end{equation}
which starts with ${\bf x}^1 = {\bf x}^0 - \alpha \nabla {\bf f}({\bf x}^0)$, any ${\bf x}^0 \in \mathbf{R}^{n\times p}$ and uses a properly bounded step size $\alpha>0$.
Extra converges at a rate $o(\frac{1}{t})$, measured by  the best running violation to the first-order optimality condition, provided that ${f}$ is Lipschitz differentiable. It improves to a linear rate of convergence if $f$ is also (restricted) strongly convex.


The subgradient-push algorithm {\cite{Nedic-SubgradientPush2015}} is proposed to solve the decentralized optimization problem \eqref{Eq:multi-agentOPT} over a time-varying directed graph.
It is a combination of the subgradient method and the push-sum protocol \cite{Kempe-Push-sumw2003,Benezit-WeightedGossip2010,Iutzeler-sum-weight2013}.
Let $A(t)$ be the mixing matrix at the $t$th iteration as defined in \eqref{Eq:MixingMatrixA*} for a time-varying directed network.
The iteration of subgradient-push is
\begin{equation}
\label{Eq:subgradient-push}
\left\{
\begin{array}{l}
{\bf z}^{t+1} = A(t){\bf z}^{t} - \alpha_t \nabla {\bf f}({\bf x}^{t}),\\
{\bf w}^{t+1} = A(t){\bf w}^{t},\\
{\bf x}^{t+1} = \diagnal({\bf w}^{t+1})^{-1}{\bf z}^{t+1},
\end{array}%
\right.
\end{equation}%
where $\alpha_t>0$ is the step size at the $t$th iteration that decays as follows: $\sum_{t=1}^{\infty} \alpha_t = \infty$, $\sum_{t=1}^{\infty} \alpha_t^2 < \infty,$ and $\alpha_t \leq \alpha_s$ for all $t>s \geq 1.$ It is shown in \cite{Nedic-SubgradientPush2015} that the convergence rate of subgradient-push algorithm is $O(\ln t/\sqrt{t})$.

\subsection{Proposed: ExtraPush}

ExtraPush combines the above two algorithms. Specifically, set arbitrary ${\bf z}^0$ and ${\bf w}^0 = {\bf 1}_{n}$; set ${\bf x}^0 = {\bf z}^0$;
for $t=1$, set ${\bf w}^1 = A{\bf w}^0$, ${\bf z}^1 = A{\bf z}^0 -\alpha \nabla {\bf f}({\bf x}^0),$
and ${\bf x}^1 = \diagnal({\bf w}^1)^{-1}{\bf z}^1$.
Letting  $\bar{A}\triangleq \frac{{\bf I}_n +A}{2}$, for $t=2,3,\ldots,$ perform
\begin{equation}
\label{Eq:EXTRA-push}
\left\{
\begin{array}{l}
{\bf z}^{t} = (A+{\bf I}_n){\bf z}^{t-1} - \bar{A}{\bf z}^{t-2} - \alpha (\nabla {\bf f}({\bf x}^{t-1}) - \nabla {\bf f}({\bf x}^{t-2})),\\
{\bf w}^{t} = A{\bf w}^{t-1},\\
{\bf x}^{t} = \diagnal({\bf w}^t)^{-1}{\bf z}^t.
\end{array}%
\right.
\end{equation}%
By the structure of $A$,  each node $i$ broadcasts its  $z_{(i)}$ to its out-neighbors at each ExtraPush iteration. The step size $\alpha>0$ needs to be properly set.
The iteration \eqref{Eq:EXTRA-push} of ExtraPush can be implemented at each agent $i$ as follows:
\begin{equation*}
\left\{
\begin{array}{l}
z_{(i)}^{t} = z_{(i)}^{t-1} + \sum_{j\in {\cal N}_i^{\mathrm{in}}}A_{ij}z_{(j)}^{t-1}   - \sum_{j\in {\cal N}_i^{\mathrm{in}}}\bar{A}_{ij}z_{(j)}^{t-2} - \alpha (\nabla f_i(x_{(i)}^{t-1}) - \nabla f_i(x_{(i)}^{t-2})),\\
w_i^{t} = \sum_{j\in {\cal N}_i^{\mathrm{in}}}A_{ij}w_j^{t-1},\\
x_{(i)}^{t} = \frac{z_{(i)}^{t}}{w_i^{t}},
\end{array}%
\right.
\end{equation*}%
where $\bar{A}_{ij}$ is the $(i,j)$th component of $\bar{A}$, and $w_i^t$ is the $i$th component of ${\bf w}^t$, for all $i,j$.

\subsection{Proposed: Normalized ExtraPush}

Normalized ExtraPush first computes the stationary distribution $\phi$ of $A$ and saves each $\phi_i$ at node $i$.
Next, in the main iteration,  the ${\bf w}$-step from {\eqref{Eq:EXTRA-push}}
is removed, and  $n\cdot\phi$ instead of ${\bf w}^{t}$ is used to obtain ${\bf x}^{t}$.
As such, the main iteration of Normalized ExtraPush simplifies \eqref{Eq:EXTRA-push}.
Letting, $$D \triangleq n\diagnal(\phi).$$
the  iteration of Normalized ExtraPush proceeds as follows:
set arbitrary ${\bf z}^0$ and ${\bf x}^0 = D^{-1}{\bf z}^0$;
for $t=1$, set ${\bf z}^1 = A{\bf z}^0 -\alpha \nabla {\bf f}({\bf x}^0)$
and ${\bf x}^1 = D^{-1}{\bf z}^1$.
For $t=2,3,\ldots,$ perform
\begin{equation}
\label{Eq:GEXTRA}
\left\{
\begin{array}{l}
{\bf z}^{t} = (A+{\bf I}_n){\bf z}^{t-1} - \bar{A}{\bf z}^{t-2} - \alpha (\nabla {\bf f}({\bf x}^{t-1}) - \nabla {\bf f}({\bf x}^{t-2})),\\
{\bf x}^{t} = D^{-1}{\bf z}^{t}.
\end{array}%
\right.
\end{equation}%
At each agent $i$, the iterate \eqref{Eq:GEXTRA} of Normalized ExtraPush can be implemented as follows:
\begin{equation*}
\left\{
\begin{array}{l}
z_{(i)}^{t} = z_{(i)}^{t-1} + \sum_{j\in {\cal N}_i^{\mathrm{in}}}A_{ij}z_{(j)}^{t-1} - \sum_{j\in {\cal N}_i^{\mathrm{in}}}\bar{A}_{ij}z_{(j)}^{t-2} - \alpha (\nabla f_i(x_{(i)}^{t-1}) - \nabla f_i(x_{(i)}^{t-2})),\\
x_{(i)}^{t} = \frac{z_{(i)}^{t}}{n \phi_i}.
\end{array}%
\right.
\end{equation*}%

Next, we present two equivalent forms of Normalized ExtraPush.
Letting ${\bf f}_{\phi}({\bf z}) \triangleq D{\bf f}(D^{-1}{\bf z}),$
we have $\nabla {\bf f}_{\phi}({\bf z}) = \nabla {\bf f}(D^{-1}{\bf z}).$
Substituting the ${\bf x}$-step of (\ref{Eq:GEXTRA}) into its ${\bf z}$-step yields the single-value iteration:
\begin{equation}
\label{Eq:GEXTRA1}
{\bf z}^{t} = (A+{\bf I}_n){\bf z}^{t-1} - \bar{A}{\bf z}^{t-2} - \alpha (\nabla {\bf f}_{\phi}({\bf z}^{t-1}) - \nabla {\bf f}_{\phi}({\bf z}^{t-2})).
\end{equation}
Upon stopping, one shall return ${\bf x}^{t} = D^{-1}{\bf z}^{t}$. The iteration ({\ref{Eq:GEXTRA1}}) is nearly identical to the Extra iteration \eqref{Eq:EXTRA} except that  \eqref{Eq:EXTRA} must use a doubly-stochastic matrix. 

Letting $A_{\phi} \triangleq D^{-1}AD$
and $\bar{A}_{\phi} \triangleq \frac{1}{2}({\bf I}_n+A_{\phi})$, which are row stochastic matrices,
gives another equivalent form of Normalized ExtraPush
\begin{equation}
\label{Eq:GEXTRA2}
{\bf x}^{t} = (A_{\phi}+{\bf I}_n){\bf x}^{t-1} - \bar{A}_{\phi}{\bf x}^{t-2} - \alpha D^{-1} (\nabla {\bf f}({\bf x}^{t-1}) - \nabla {\bf f}({\bf x}^{t-2})),
\end{equation}
which, compared to the Extra iteration \eqref{Eq:EXTRA}, has the extra diagonal matrix $D^{-1}$. Indeed, this iteration generalizes Extra to use row-stochastic  matrices $A_{\phi}$ and $\bar{A}$.

\section{Preliminary Analysis of ExtraPush}

In this section, we first develop the first-order optimality conditions for the problem \eqref{Eq:consensusProblem}
and then provide the convergence of ExtraPush under the boundedness assumption.

\begin{Thm}
\label{Thm:1stOrderOptCond}
{\bf (first-order optimality conditions)}
Suppose that graph ${\cal G}$ is strongly connected.
Then ${\bf x}^*$ is consensual and $x^*_{(1)} \equiv x^*_{(2)} \equiv \cdots \equiv x^*_{(n)}$ is an optimal solution of $(\ref{Eq:multi-agentOPT})$
if and only if, for some $\alpha>0$, there exist ${\bf z}^* \in \nullspc({\bf I}_n - A)$ and ${\bf y}^* \in \nullspc({\bf 1}_n^T)$
such that the following conditions hold
\begin{equation}
\label{Eq:1stOrderOpt}
\left\{
\begin{array}{l}
{\bf y}^* + \alpha \nabla {\bf f}({\bf x}^*) = 0,\\
{\bf x}^* = D^{-1} {\bf z}^*.
\end{array}%
\right.
\end{equation}%
{\rm(}We let ${\cal L}^*$ denote the set of triples $({\bf z}^*, {\bf y}^*, {\bf x}^*)$ satisfying the above conditions.{\rm)}
\end{Thm}

\begin{proof}
Assume that ${\bf x}^*$ is consensual and $x^*_{(1)} \equiv x^*_{(2)} \equiv \cdots \equiv x^*_{(n)}$ is optimal.
Let ${\bf z}^* = n \diagnal(\phi) {\bf x}^* = n( \phi x^{*T}_{(1)}).$
Then $\phi {\bf 1}_n^T {\bf z}^* = \phi {\bf 1}_n^T n \phi x^{*T}_{(1)} = n \phi x^{*T}_{(1)}= {\bf z}^*.$
It implies that ${\bf z}^* \in \nullspc({\bf I} - \phi {\bf 1}_n^T)$.
By Property {\ref{Prop:A}} (ii), it follows that ${\bf z}^* \in \nullspc({\bf I}_n - A)$.
Moreover, letting ${\bf y}^* = -\alpha \nabla {\bf f}({\bf x}^*),$ it holds that ${\bf 1}_n^T {\bf y}^* = -\alpha {\bf 1}_n^T \nabla {\bf f}({\bf x}^*) = 0$, that is, ${\bf y}^* \in \nullspc({\bf 1}_n^T)$.

On the other hand, assume (\ref{Eq:1stOrderOpt}) holds. By Property {\ref{Prop:A}} (ii), it follows that ${\bf z}^* = \phi {\bf 1}_n^T {\bf z}^*.$
Plugging ${\bf x}^* = D^{-1} {\bf z}^*$ gives ${\bf x}^* = \frac{1}{n}{\bf 1}_n{\bf 1}_n^T {\bf z}^*,$
which implies that ${\bf x}^*$ is consensual. Moreover, by ${\bf y}^* + \alpha \nabla {\bf f}({\bf x}^*) = 0$
and ${\bf y}^* \in \nullspc({\bf 1}_n^T)$, it holds ${\bf 1}_n^T \nabla {\bf f}({\bf x}^*) = -\frac{1}{\alpha}{\bf 1}_n^T{\bf y}^* =0$,
which implies that ${\bf x}^*$ is optimal.\hfill$\Box$
\end{proof}

Introducing the sequence
\begin{equation}
\label{Eq:ut}
{\bf y}^t \triangleq \sum_{k=0}^t (\bar{A}-A){\bf z}^k,
\end{equation}
the iteration (\ref{Eq:EXTRA-push}) of ExtraPush can be rewritten as
\begin{equation}
\left\{
\begin{array}{l}
\bar{A}{\bf z}^{t+1} = \bar{A}{\bf z}^t -\alpha \nabla {\bf f}({\bf x}^t) - {\bf y}^{t+1},\\
{\bf y}^{t+1} = {\bf y}^t + (\bar{A}-A){\bf z}^{t+1},\\
{\bf w}^{t+1} = A{\bf w}^{t},\\
{\bf x}^{t+1} = \diagnal({\bf w}^{t+1})^{-1}{\bf z}^{t+1}.
\end{array}%
\right.
\label{Eq:Var1ofEXTRAPush}
\end{equation}

\begin{Thm}
\label{Thm:LimitPoint_Optimality}
Suppose that the sequence $\{{\bf x}^t\}$ generated by ExtraPush {\eqref{Eq:EXTRA-push}} and the sequence $\{{\bf y}^t\}$ defined in {\eqref{Eq:ut}} are bounded. Then, any limit point of $\{({\bf z}^t, {\bf y}^t, {\bf x}^t)\}$, denoted by $({\bf z}^*, {\bf y}^*, {\bf x}^*)$, satisfies the optimality conditions $(\ref{Eq:1stOrderOpt})$.
\end{Thm}

\begin{proof}
By Property {\ref{Prop:A}},  $\{{\bf w}^t\}$ is bounded. By the last update of \eqref{Eq:Var1ofEXTRAPush} and the boundedness of both $\{{\bf x}^t\}$ and $\{{\bf w}^t\}$, $\{{\bf z}^t\}$ is bounded.
Hence, there exists a convergent subsequence $\{({\bf z}, {\bf y}, {\bf w}, {\bf x})^{t_j}\}_{j=1}^{\infty}$.
Let $({\bf z}^*, {\bf y}^*, {\bf w}^*, {\bf x}^*)$ be its limit.
By {\eqref{Eq:LimofAt}}, we know that ${\bf w}^* = n \phi$ and thus that ${\bf x}^* = D^{-1} {\bf z}^*.$
Letting $t\to\infty$ in the second equation of (\ref{Eq:Var1ofEXTRAPush}) gives ${\bf z}^* = A{\bf z}^*$, or equivalently ${\bf z}^* \in \nullspc({\bf I}_n - A)$.
Similarly, letting  $t\to\infty$ in the first equation of (\ref{Eq:Var1ofEXTRAPush}) yields ${\bf y}^* + \alpha \nabla {\bf f}({\bf x}^*) = 0.$
Moreover, from the definition {\eqref{Eq:ut}} of ${\bf y}^t$ and the facts that both $A$ and ${\bar{A}}$ are column stochastic, it follows that ${\bf 1}_n^T {\bf y}^* =0$ and ${\bf 1}_n^T \nabla {\bf f}({\bf x}^*) =0$.
Therefore, $({\bf z}^*, {\bf y}^*, {\bf x}^*)$ satisfies the optimality conditions (\ref{Eq:1stOrderOpt}).\hfill$\Box$
\end{proof}

\section{Convergence of Normalized ExtraPush}

In this section, we show the linear convergence of Normalized ExtraPush under the smoothness and strong convexity assumptions of the objective function.
Similar to (\ref{Eq:Var1ofEXTRAPush}), introducing a new sequence ${\bf y}^t = \sum_{k=0}^t (\bar{A}-A){\bf z}^k,$
 the iterative formula (\ref{Eq:GEXTRA}) of Normalized ExtraPush implies
\begin{equation}
\left\{
\begin{array}{l}
\bar{A}{\bf z}^{t+1} = \bar{A}{\bf z}^t -\alpha \nabla {\bf f}({\bf x}^t) - {\bf y}^{t+1},\\
{\bf y}^{t+1} = {\bf y}^t + (\bar{A}-A){\bf z}^{t+1},\\
{\bf x}^{t+1} = D^{-1} {\bf z}^{t+1}.
\end{array}%
\right.
\label{Eq:GEXTRA0}
\end{equation}

\begin{Thm}
\label{Thm:LimitPoint_GEXTRA}
Suppose that the sequence $\{{\bf x}^t\}$ generated by Normalized ExtraPush {\eqref{Eq:GEXTRA}} is bounded,
and that the sequence $\{{\bf y}^t\}$ is also bounded.
Then, any limit point of $\{({\bf z}^t, {\bf y}^t, {\bf x}^t)\}_{t=0}^{\infty}$, denoted by $({\bf z}^*, {\bf y}^*, {\bf x}^*)$,
satisfies the first-order optimality conditions $(\ref{Eq:1stOrderOpt})$.
\end{Thm}

The proof is very similar to that of Theorem {\ref{Thm:LimitPoint_Optimality}}.
It only needs to replace the sequence $\{{\bf w}^t\}$ with its limitation $n\phi$ in the proof procedure,
thus we omit it here.
From Theorem {\ref{Thm:LimitPoint_GEXTRA}}, it shows that Normalized ExtraPush
has subsequence convergence to an optimal solution of the considered optimization problem under the boundedness assumption.
To obtain the linear convergence of Normalized ExtraPush, we still need the following assumptions.

\begin{Assump}
\label{Assump:ExisOptSolution}
{\bf (existence of solution)}
Let ${\cal X}^*$ be the optimal solution set of problem \eqref{Eq:multi-agentOPT}, and   assume that  ${\cal X}^*$ is nonempty.
\end{Assump}

\begin{Assump}
\label{Assump:ObjFun}
For each agent $i$, its objective function $f_i$ satisfies the following:
\begin{enumerate}
\item[(i)] {\bf (Lipschitz differentiability)}
$f_i$ is  differentiable, and its gradient $\nabla f_i$ is $L_i$-Lipschitz continuous, i.e.,
$\|\nabla f_i(x)-\nabla f_i(y)\|\leq L_i \|x-y\|, \forall x,y\in \mathbf{R}^p$;

\item[(ii)] {\bf (quasi-strong convexity)}
$f_i$ is quasi-strongly convex, and there exists a positive constant $S_i$
such that $S_i \|x^*-y\|^2 \leq \langle \nabla f_i(x^*)-\nabla f_i(y), x^*-y \rangle$ for any $y\in \mathbf{R}^p$
and some optimal value $x^* \in {\cal X}^*$.
\end{enumerate}
\end{Assump}

Following Assumption {\ref{Assump:ObjFun}}, there hold for any ${\bf x}, {\bf y} \in \mathbf{R}^{n\times p}$ and some ${\bf x}^* \equiv {\bf 1}_n {(x^*)}^T$
\begin{align}
&\|\nabla {\bf f}({\bf x}) - \nabla {\bf f}({\bf y})\| \leq L_f \|{\bf x}-{\bf y}\|,\label{Eq:LipC-Grad}\\
&S_f \|{\bf x}^*-{\bf y}\|^2 \leq \langle \nabla {\bf f}({\bf x}^*)-\nabla {\bf f}({\bf y}), {\bf x}^*-{\bf y}\rangle,  \label{Eq:StrongCVX}
\end{align}
where the constants $L_f \triangleq \max_i L_i$ and $S_f \triangleq \min_i S_i$.

\begin{Assump}
\label{Assump:WeightingMatrices}
{\bf (positive definiteness)}
$D^{-1}\bar{A}+\bar{A}^TD^{-1} \succ 0$.
\end{Assump}

By noticing $D^{-1}\bar{A}+\bar{A}^TD^{-1} = D^{-1/2}(D^{-1/2}\bar{A}D^{1/2} + D^{1/2}\bar{A}^TD^{-1/2})D^{-1/2}$, we can guarantee the positive definiteness of $D^{-1}\bar{A}+\bar{A}^TD^{-1}$ by ensuring the matrix $\bar{A}+\bar{A}^T$ to be positive definite.
Note that $\bar{A}_{ii} > \sum_{j\neq i} \bar{A}_{ij}$ for each $i$, which means that $\bar{A}$ is strictly column-diagonal dominant.
To ensure the positive definiteness of $\bar{A}+\bar{A}^T$, each node $j$ can be ``selfish" and take a sufficiently large $A_{jj}$.

Before presenting the main result, we introduce the following notation. For each $t$, introducing ${\bf u}^t = \sum_{k=0}^t {\bf z}^k$,
then the Normalized ExtraPush iteration {\eqref{Eq:GEXTRA}} reduces to
\begin{equation}
\left\{
\begin{array}{l}
\bar{A}{\bf z}^{t+1} = \bar{A}{\bf z}^t -\alpha \nabla {\bf f}({\bf x}^t) - (\bar{A}-A){\bf u}^{t+1}\\
{\bf u}^{t+1} = {\bf u}^t + {\bf z}^{t+1}\\
{\bf x}^{t+1} = D^{-1} {\bf z}^{t+1}.
\end{array}%
\right.
\label{Eq:GEXTRA3}
\end{equation}
Let $({\bf z}^*, {\bf y}^*, {\bf x}^*) \in {\cal L}^*$, where ${\bf x}^*$ has been specified in {\eqref{Eq:StrongCVX}}.
Let ${\bf u}^*$ be any matrix that satisfies $(\bar{A}-A){\bf u}^* = {\bf y}^*$.
For simplicity, we introduce
\begin{equation}
\label{Eq:MetricForm}
{\bf v}^t = \left(
\begin{array}{c}
{\bf z}^t\\
{\bf u}^t
\end{array}
\right),~
{\bf v}^* = \left(
\begin{array}{c}
{\bf z}^*\\
{\bf u}^*
\end{array}
\right),~
G = \left(
\begin{array}{cc}
N^T & {\bf 0}\\
{\bf 0} & M
\end{array}
\right),
~S = \left(
\begin{array}{cc}
{\bf 0} & M\\
 -M^T & {\bf 0}
\end{array}
\right),
\end{equation}
where $N = D^{-1}\bar{A}$, $M =  D^{-1}(\bar{A}-A)$.
Let ${\bf f}_D({\bf z}) \triangleq  {\bf f}(D^{-1}{\bf z})$ and $\bar{{\bf f}}({\bf v}) \triangleq {\bf f}_D({\bf z})$.
Then $\nabla \bar{{\bf f}}({\bf v}) = [\nabla {\bf f}_D({\bf z}), 0]$.
By {\eqref{Eq:LipC-Grad}} and {\eqref{Eq:StrongCVX}}, there hold
\begin{align}
&\|\nabla \bar{{\bf f}}({\bf v}_1) - \nabla \bar{{\bf f}}({\bf v}_2) \| = \|\nabla {\bf f}_D({\bf z}_1) - \nabla {\bf f}_D({\bf z}_2) \| \leq \bar{L} \|{\bf z}_1 - {\bf z}_2\|,\label{Eq:LipC-Grad-fD} \\
&\bar{\mu} \|{\bf z}^*-{\bf z}\|^2 \leq \langle \nabla {\bf f}_D({\bf z}^*)-\nabla {\bf f}_D({\bf z}), {\bf z}^*-{\bf z}\rangle=\langle \nabla \bar{{\bf f}}({\bf v}^*)-\nabla \bar{{\bf f}}({\bf v}), {\bf v}^*-{\bf v}\rangle,\label{Eq:StrongCVX-FD}
\end{align}
where $\bar{L} \triangleq \frac{L_f}{\sigma^2_{\min}(D)}$ and $\bar{\mu} \triangleq \frac{S_f}{\sigma^2_{\max}(D)}$.
By {\eqref{Eq:GEXTRA3}} and {\eqref{Eq:MetricForm}}, the Normalized ExtraPush iteration {\eqref{Eq:GEXTRA}} implies
\begin{align}
\label{Eq:GEXTRA-MatrixForm}
G^T({\bf v}^{t+1} - {\bf v}^t) = -S{\bf v}^{t+1} -\alpha \nabla \bar{{\bf f}}({\bf v}^t).
\end{align}

Next, we will show that $G+G^T$ is positive semidefinite, which by Assumption {\ref{Assump:WeightingMatrices}} implies that $N+N^T$ is positive definite. It is sufficient to show that $M+M^T$ is positive semidefinite.
Note that
\begin{align*}
&\ M+M^T = \frac{D^{-1}({\bf I}_n-A)}{2} +\frac{({\bf I}_n-A)^TD^{-1}}{2} \\
 =&\ D^{-1/2}\bigg({\bf I}_n - \frac{D^{1/2}A^TD^{-1/2}+D^{-1/2}AD^{1/2}}{2}\bigg)D^{-1/2} \triangleq D^{-1/2}\Lambda D^{-1/2},
\end{align*}
and by Property {\ref{Prop:A}} (iii), $n\phi^T$ is the left eigenvector of $A^T$ corresponding to eigenvalue $1$,
and thus, $\Lambda$ is  the Laplacian of a certain directed graph ${\cal G}'$ with $A^T$ being its corresponding transition probability matrix {\cite{ChungLaplace2005}}.
It follows that $0=\lambda_1(\Lambda) \leq \lambda_2(\Lambda) \leq \cdots \leq  \lambda_{n}(\Lambda)$, where $\lambda_i(\Lambda)$ denotes the $i$th eigenvalue of $\Lambda$.
Therefore,
$M+M^T$ is positive semidefinite, and the following property holds
\[
\|x\|_{G}^2 = \frac{1}{2}\|x\|_{G+G^T}^2\geq 0, \quad\forall x\in \mathbf{R}^{n}.
\]
Let $$c_1 = \frac{\lambda_{\max}(MM^T)}{\tilde{\lambda}_{\min}(M^TM)},\;\;\;
c_2 = \frac{\lambda_{\max}(\frac{M+M^T}{2})}{\tilde{\lambda}_{\min}(M^TM)},\;\;\;
c_3 = \lambda_{\max}(NN^T) + 3c_1\lambda_{\max}(N^TN),$$
and let $$\Delta_1 = \bigg(\bar{\mu} -\frac{\eta}{2}\bigg)^2 - 6c_1{\bar L}^2,\;\;\;
\Delta_2 = \frac{{\bar L}^4}{4\eta^2} -3c_1{\bar L}^2\sigma \Big(c_3\sigma - \lambda_{\min}(N^T+N)\Big)$$ for some appropriate tunable parameters $\eta$ and $\sigma$.
Then we describe our main result as follows.

\begin{Thm}
\label{Thm:ConvThm}
Under Assumptions {\rm{\ref{Assump:Network}}}-{\rm{\ref{Assump:WeightingMatrices}}},
if the step size parameter $\alpha$ satisfies
\begin{align}
\label{Eq:Condonalpha*}
\frac{{\bar \mu} - \frac{\eta}{2} - \sqrt{\Delta_1}}{3c_1{\bar L}^2\sigma} < \alpha < \min\Bigg\{\frac{{\bar \mu} - \frac{\eta}{2} + \sqrt{\Delta_1}}{3c_1{\bar L}^2\sigma}, \frac{-\frac{{\bar L}^2}{2\eta}+ \sqrt{\Delta_2}}{3c_1{\bar L}^2 \sigma}\Bigg\}
\end{align}
for some appropriate $\eta$ and $\sigma$ as specified in {\eqref{Eq:Condoneta}} and {\eqref{Eq:Condonsigma}}, respectively, then the sequence $\{{\bf v}^t\}$ defined in {\eqref{Eq:MetricForm}} satisfies
\begin{equation}
\label{Eq:LinConvThm}
\|{\bf v}^{t}-{\bf v}^*\|_G^2 \geq (1+\delta)\|{\bf v}^{t+1}-{\bf v}^*\|_G^2,
\end{equation}
for  $\delta>0$ obeying
\begin{align*}
0<\delta \leq
\min\Bigg\{ \frac{-\frac{1}{\sigma} +(\bar{\mu} - \frac{\eta}{2})\alpha - \frac{3}{2}c_1{\bar L}^2\sigma \alpha^2}{\lambda_{\max}(\frac{N+N^T}{2})+ 3c_2\alpha^2{\bar L}^2},
\frac{\lambda_{\min}(\frac{N^T+N}{2}) - \frac{c_3\sigma}{2} - \frac{{\bar L}^2\alpha}{2\eta} - \frac{3}{2}c_1{\bar L}^2\sigma \alpha^2}{3c_2(\lambda_{\max}(N^TN)+\alpha^2{\bar L}^2)}\Bigg\}.
\end{align*}
\end{Thm}

From this theorem, the sequence $\{{\bf v}^t\}$ converges to ${\bf v}^*$ at a linear rate in the sense of ``$G$-norm''.
By the definition of ${\bf v}^*$ in {\eqref{Eq:MetricForm}}, ${\bf v}^*$ is indeed defined by some optimal value $({\bf z}^*, {\bf y}^*, {\bf x}^*)$.
Roughly speaking, bigger $\delta$ means faster convergence rate. As specified in Theorem {\ref{Thm:ConvThm}}, $\delta$ is affected by many factors.
Generally, $\delta$ decreases with respect to both $\lambda_{\max}(\frac{N+N^T}{2})$ and $\lambda_{\max}(N^TN)$,
which potentially implies that if all nodes are more ``selfish", that is, they hold more information for themselves than sending  to their out-neighbors.
Consequently, the information mixing speed of the network will get smaller, and thus the convergence of  Normalized ExtraPush becomes slower.
Therefore, we suggest a more democratic rule (such as the matrix $A$ specified in {\eqref{Eq:MixingMatrixA*}})  for faster convergence in practice.
To ensure $\delta>0$, it requires that the step size $\alpha$  lie in an appropriate interval.
It should be pointed out that the condition \eqref{Eq:Condonalpha*} on $\alpha$ is sufficiently, not necessary, for the linear convergence of Normalized ExtraPush. Normalized ExtraPush algorithm may not diverge if a small $\alpha$ is set. In fact, in the next section, it can be observed that both ExtraPush and Normalized ExtraPush algorithms  converge under small values of $\alpha$. In general, a smaller $\alpha$ implies a slower rate of convergence.

To prove Theorem {\ref{Thm:ConvThm}}, we  need the following lemmas.
\begin{Lemm}
For any  $({\bf z}^*, {\bf y}^*, {\bf x}^*) \in {\cal L}^*$, let ${\bf u}^*$ satisfy $(\bar{A}-A){\bf u}^* = {\bf y}^*$. Then there hold
\begin{align}
& M{\bf z}^* = {\bf 0}, \label{Eq:z*1}\\
& M^T {\bf z}^* = {\bf 0}, \label{Eq:z*2} \\
& S{\bf v}^* + \alpha \nabla \bar{\bf f}({\bf v}^*)=0. \label{Eq:optcond-v*}
\end{align}
\end{Lemm}
\begin{proof}
By the optimality of $({\bf z}^*, {\bf y}^*, {\bf x}^*)$, the followings hold: (i) $(\bar{A}-A){\bf z}^* =0$, and thus $ M{\bf z}^* =0$; (ii) $D^{-1}{\bf z}^* = {\bf x}^*$ is consensual; from the column stochasticity of both $A$ and $\bar{A}$,  it follows $M^T {\bf z}^* = (\bar{A}-A)^T {\bf x}^* = 0$; (iii) $M{\bf u}^* + \alpha \nabla {\bf f}_D({\bf z}^*) = D^{-1}{\bf y}^* + \alpha D^{-1}\nabla {\bf f}({\bf x}^*)=0,$  with $M^T {\bf z}^*=0$, which imply $S{\bf v}^* + \alpha \nabla \bar{\bf f}({\bf v}^*) =0.$\hfill$\Box$
\end{proof}

\begin{Lemm}
\label{Lemm:IterativeRelation}
For any $t\in \mathbf{N}$, it holds
\begin{equation}
\label{Eq:IterativeRelation}
N({\bf z}^{t+1} - {\bf z}^t) = - M({\bf u}^{t+1} - {\bf u}^*) - \alpha [\nabla {\bf f}_D({\bf z}^t) - \nabla {\bf f}_D({\bf z}^*)].
\end{equation}
\end{Lemm}
This lemma follows from {\eqref{Eq:GEXTRA3}} and the fact $M{\bf u}^* + \alpha \nabla {\bf f}_D({\bf z}^*) = 0$  in the last lemma.

\begin{Lemm}
\label{Lemm:DiffofTwoIter}
Let $\{{\bf v}^t\}$ be a sequence generated by the iteration {\eqref{Eq:GEXTRA-MatrixForm}} and ${\bf v}^*$ be defined in {\eqref{Eq:MetricForm}}. Then, it holds
\begin{align}
\label{Eq:DiffofTwoIter}
\|{\bf v}^{t+1} - {\bf v}^*\|_G^2 - \|{\bf v}^t - {\bf v}^*\|^2_G
& \leq -\|{\bf v}^{t+1} - {\bf v}^t\|^2_G + \|{\bf z}^{t+1} - {\bf z}^t\|^2_{\frac{\sigma}{2}NN^T + \frac{\alpha \bar{L}^2}{2\eta}{\bf I}_n} \nonumber\\
&\qquad\;\;   - \|{\bf z}^{*} - {\bf z}^{t+1}\|^2_{(\alpha \bar{\mu}-\frac{\alpha \eta}{2}-\frac{1}{\sigma}){\bf I}_n} + \frac{\sigma}{2}\|{\bf u}^{*} - {\bf u}^{t+1}\|^2_{MM^T},
\end{align}
where $\sigma, \eta>0$ are two tunable parameters.
\end{Lemm}
\begin{proof}
Note that
\begin{align}
\label{Eq:SquareEqForm}
 \|{\bf v}^{t+1} - {\bf v}^*\|_G^2 - \|{\bf v}^t - {\bf v}^*\|^2_G
&= -\|{\bf v}^{t+1} - {\bf v}^t\|^2_G + \langle {\bf v}^*- {\bf v}^{t+1}, G({\bf v}^{t}-{\bf v}^{t+1}) \rangle\nonumber\\
& \qquad\;\;    + \langle {\bf v}^*- {\bf v}^{t+1}, G^T({\bf v}^{t}-{\bf v}^{t+1}) \rangle.
\end{align}
In the following, we analyze the two inner-product terms:
\begin{align}
\label{Eq:1stTerm}
\langle {\bf v}^*- {\bf v}^{t+1}, G({\bf v}^{t}-{\bf v}^{t+1}) \rangle
& = \langle {\bf z}^*- {\bf z}^{t+1}, N^T({\bf z}^{t}-{\bf z}^{t+1}) \rangle + \langle M^T({\bf u}^*- {\bf u}^{t+1}), {\bf u}^{t}-{\bf u}^{t+1} \rangle \nonumber\\
(\because {\eqref{Eq:z*1}}, M{\bf z}^* = {\bf 0})
& = \langle {\bf z}^*- {\bf z}^{t+1}, N^T({\bf z}^{t}-{\bf z}^{t+1}) \rangle + \langle M^T({\bf u}^*- {\bf u}^{t+1}), {\bf z}^{*}-{\bf z}^{t+1} \rangle \nonumber\\
& \leq \frac{\sigma}{2}\|{\bf z}^{t}-{\bf z}^{t+1}\|_{NN^T}^2 + \frac{1}{\sigma}\|{\bf z}^{*}-{\bf z}^{t+1}\|^2 + \frac{\sigma}{2}\|{\bf u}^{*}-{\bf u}^{t+1}\|_{MM^T}^2,
\end{align}
and
\begin{align}
\label{Eq:2ndTerm}
\langle {\bf v}^*- {\bf v}^{t+1}, G^T({\bf v}^{t}-{\bf v}^{t+1}) \rangle
& = \langle {\bf v}^*- {\bf v}^{t+1}, S{\bf v}^{t+1}+\alpha \nabla \bar{\bf f}({\bf v}^{t}) \rangle  \ (\because {\eqref{Eq:MetricForm}})\nonumber\\
& = \langle {\bf v}^*- {\bf v}^{t+1}, S({\bf v}^{t+1}-{\bf v}^{*})+\alpha (\nabla \bar{\bf f}({\bf v}^{t})-\nabla \bar{\bf f}({\bf v}^{*})) \rangle \ (\because {\eqref{Eq:optcond-v*}}) \nonumber\\
(\because S = -S^T) \ & = \alpha \langle {\bf v}^*- {\bf v}^{t+1}, \nabla \bar{\bf f}({\bf v}^{t})-\nabla \bar{\bf f}({\bf v}^{*})\rangle \nonumber\\
& = \alpha \langle {\bf v}^*- {\bf v}^{t+1}, \nabla \bar{\bf f}({\bf v}^{t+1})-\nabla \bar{\bf f}({\bf v}^{*})\rangle\nonumber\\
& \quad + \alpha \langle {\bf v}^*- {\bf v}^{t+1}, \nabla \bar{\bf f}({\bf v}^{t})-\nabla \bar{\bf f}({\bf v}^{t+1})\rangle \nonumber\\
&\leq -\alpha \bar{\mu} \|{\bf z}^{t+1} - {\bf z}^*\|^2 + \frac{\alpha \eta}{2} \|{\bf z}^* - {\bf z}^{t+1}\|^2 + \frac{\alpha \bar{L}^2}{2\eta} \|{\bf z}^t - {\bf z}^{t+1}\|^2.
\end{align}
Substituting {\eqref{Eq:1stTerm}} and {\eqref{Eq:2ndTerm}} into {\eqref{Eq:SquareEqForm}}, then we can conclude the lemma.\hfill$\Box$
\end{proof}

\begin{proof} {(for Theorem {\ref{Thm:ConvThm}})}
In order to establish {\eqref{Eq:LinConvThm}} for some constant $\delta>0$, in light of Lemma {\ref{Lemm:DiffofTwoIter}}, it is sufficient to show that the right-hand side of {\eqref{Eq:DiffofTwoIter}} is no more than $-\delta \|{\bf v}^{t+1} - {\bf v}^{*}\|_G^2$,
which implies
\begin{align}
\label{Eq:AimIneq-1}
\|{\bf z}^{t+1} - {\bf z}^{*}\|_P^2 + \|{\bf z}^{t+1} - {\bf z}^{t}\|_Q^2 \geq \|{\bf u}^{t+1} - {\bf u}^{*}\|_R^2,
\end{align}
where
\begin{align*}
&P = \bigg(\alpha \bar{\mu} - \frac{\alpha \eta}{2} -\frac{1}{\sigma}\bigg){\bf I}_n - \delta \frac{N+N^T}{2},\;\;\;
Q = \frac{N^T+N}{2} - \frac{\sigma}{2}NN^T - \frac{\alpha \bar{L}^2}{2\eta}{\bf I}_n,\\
&R = \frac{\sigma}{2}MM^T + \delta {\bigg(\frac{M+M^T}{2}\bigg)}.
\end{align*}

\textit{Establishing {\eqref{Eq:AimIneq-1}}: Step 1.} From Lemma {\ref{Lemm:IterativeRelation}}, there holds
\begin{align}
\label{Eq:ut-ustar1}
\|{\bf u}^{*} - {\bf u}^{t+1}\|_{M^TM}^2
& = \|M({\bf u}^{*} - {\bf u}^{t+1})\|^2 \nonumber\\
& = \|N({\bf z}^{t+1} - {\bf z}^{t}) + \alpha [\nabla {\bf f}_D({\bf z}^{t+1}) - \nabla {\bf f}_D({\bf z}^*)] + \alpha [\nabla {\bf f}_D({\bf z}^{t}) - \nabla {\bf f}_D({\bf z}^{t+1})]\|^2 \nonumber\\
& \leq 3\|{\bf z}^{t+1} - {\bf z}^{t}\|_{N^TN}^2 + 3\alpha^2 \bar{L}^2 \|{\bf z}^{t+1} - {\bf z}^{*}\|^2 + 3\alpha^2  {\bar L}^2 \|{\bf z}^{t+1} - {\bf z}^{t}\|^2 \nonumber\\
& = \|{\bf z}^{t+1} - {\bf z}^{*}\|_{3\alpha^2\bar{L}^2{\bf I}_n}^2 + \|{\bf z}^{t+1} - {\bf z}^{t}\|_{3(N^TN+\alpha^2{\bar L}^2{\bf I}_n)}^2.
\end{align}
Note that
\begin{align*}
\frac{\|{\bf u}^{*} - {\bf u}^{t+1}\|_{\frac{\sigma}{2}MM^T+\delta M}^2}{\frac{\sigma \lambda_{\max}(MM^T)}{2}+\delta \lambda_{\max}(\frac{M+M^T}{2})}
\leq \|{\bf u}^{*} - {\bf u}^{t+1}\|^2
\leq \frac{\|{\bf u}^{*} - {\bf u}^{t+1}\|_{M^TM}^2}{\tilde{\lambda}_{\min}(M^TM)}.
\end{align*}
If the following conditions hold
\begin{equation}
\label{Eq:AimIneq-2}
\left\{
\begin{array}{l}
P \succeq 3(\frac{1}{2}c_1\sigma + c_2 \delta)\alpha^2 {\bar L}^2 {\bf I}_n,\\
Q \succeq 3(\frac{1}{2}c_1\sigma + c_2 \delta)(N^TN + \alpha^2 {\bar L}^2 {\bf I}_n),
\end{array}%
\right.
\end{equation}%
then {\eqref{Eq:AimIneq-1}} holds.
To show {\eqref{Eq:AimIneq-2}}, it is sufficient to prove
\begin{equation}
\label{Eq:AimIneq-3}
\left\{
\begin{array}{l}
(\lambda_{\max}(\frac{N+N^T}{2})+ 3c_2\alpha^2{\bar L}^2)\delta \leq -\frac{1}{\sigma} +(\bar{\mu} - \frac{\eta}{2})\alpha - \frac{3}{2}c_1{\bar L}^2\sigma \alpha^2,\\
3c_2(\lambda_{\max}(N^TN)+\alpha^2{\bar L}^2) \delta \leq \lambda_{\min}(\frac{N^T+N}{2}) - \frac{c_3\sigma}{2} - \frac{{\bar L}^2\alpha}{2\eta} - \frac{3}{2}c_1{\bar L}^2\sigma \alpha^2.
\end{array}%
\right.
\end{equation}%
Let
$c_4 \triangleq (\bar{\mu}- \frac{\eta}{2})+\sqrt{\Delta_1}$,
$c_5 \triangleq \frac{{\bar L}^2}{\eta}$,
$c_{6} \triangleq \frac{2c_4c_5 + 12c_1{\bar L}^2}{c_4^2}$,
$c_{7} \triangleq \frac{\lambda_{\min}^2(N^T+N)}{4c_3}$,
$c_{8} \triangleq a(c_{7}+2)-(2-c_{7})$ for some positive constant $a\in (0,1)$,
$\Delta_3 \triangleq \lambda_{\min}^2(N^T+N) - 4c_3c_{6}$.
After reduction, we  claim that if the following conditions hold
\begin{align}
&\frac{2-c_{7}}{2+c_{7}}<a<1,\label{Eq:Cond-a}\\
&{\bar \mu}
> \bigg(\sqrt{\frac{6c_1}{1-a^2}} + \frac{1}{c_8}\sqrt{\frac{1-a^2}{6c_1}}\bigg) \bar{L},\label{Eq:CondonFun}\\
&{\bar \mu}\bigg(1-\sqrt{1-\frac{4{\bar L}^2}{c_8{\bar \mu}^2}}\bigg) < \eta < \min \bigg\{{\bar \mu}\bigg(1+\sqrt{1-\frac{4{\bar L}^2}
{c_8{\bar \mu}^2}}\bigg), 2\bigg(\bar{\mu} - \sqrt{\frac{6c_1}{1-a^2}}\bar{L}\bigg)\bigg\},\label{Eq:Condoneta}\\
&\frac{\lambda_{\min}(N^T+N) - \sqrt{\Delta_3}}{2c_3}< \sigma <\frac{\lambda_{\min}(N^T+N) + \sqrt{\Delta_3}}{2c_3},\label{Eq:Condonsigma}\\
&\frac{{\bar \mu} - \frac{\eta}{2} - \sqrt{\Delta_1}}{3c_1{\bar L}^2\sigma} < \alpha < \min\bigg\{\frac{{\bar \mu} - \frac{\eta}{2} + \sqrt{\Delta_1}}{3c_1{\bar L}^2\sigma}, \frac{-\frac{{\bar L}^2}{2\eta}+ \sqrt{\Delta_2}}{3c_1{\bar L}^2 \sigma}\bigg\},\label{Eq:Condonalpha}
\end{align}
then {\eqref{Eq:AimIneq-3}} holds for some positive constant $\delta$.
We then end the proof of this theorem.\hfill$\Box$
\end{proof}

\section{Numerical Experiments}

In this section, we present the results of a series of numerical experiments that demonstrate the effectiveness of the proposed algorithms relative to the subgradient-push algorithm. The used network and its corresponding mixing matrix $A$ are depicted in Fig. {\ref{Fig:Network}}.

\subsection{Decentralized Least Squares}

Consider the following decentralized least squares  problem:
\begin{equation}
\label{Eq:DistributedLS}
x^* \gets \mathop{\mathrm{argmin}}_{x\in \mathbf{R}^p} f(x) = \sum_{i=1}^n f_i(x),
\end{equation}
where $f_i(x) = \frac{1}{2} \|B_{(i)} x-b_{(i)}\|_2^2, B_{(i)} \in \mathbf{R}^{m_i\times p}, b_{(i)} \in \mathbf{R}^{m_i}$ for $i=1, \ldots,n.$ The  solution $x^*$ is $B^{\dag}b$, where $B = \sum_{i=1}^n B_{(i)}^TB_{(i)}$, $b = \sum_{i=1}^n B_{(i)}^T b_{(i)}$, and $B^{\dag}$ is the pseudo-inverse of $B$.
In this experiment, we take $n=5$, $p=256$, and $m_i = 100$ for $i=1,\ldots,5$.
For both ExtraPush and Normalized ExtraPush, we first choose an $\alpha$ in the hand-optimized manner (in this case, $\alpha = 0.1$) and then take a smaller one like $\alpha = 0.02$ to show the difference due to a smaller step size.
The step size of the subgradient-push algorithm is handed optimized to $\alpha_t = \frac{0.8}{\sqrt{t}}$.
The experiment results are illustrated in Fig. {\ref{Fig:DLS}}.

\begin{figure}[ht]
\centering
  \includegraphics[scale=0.6]{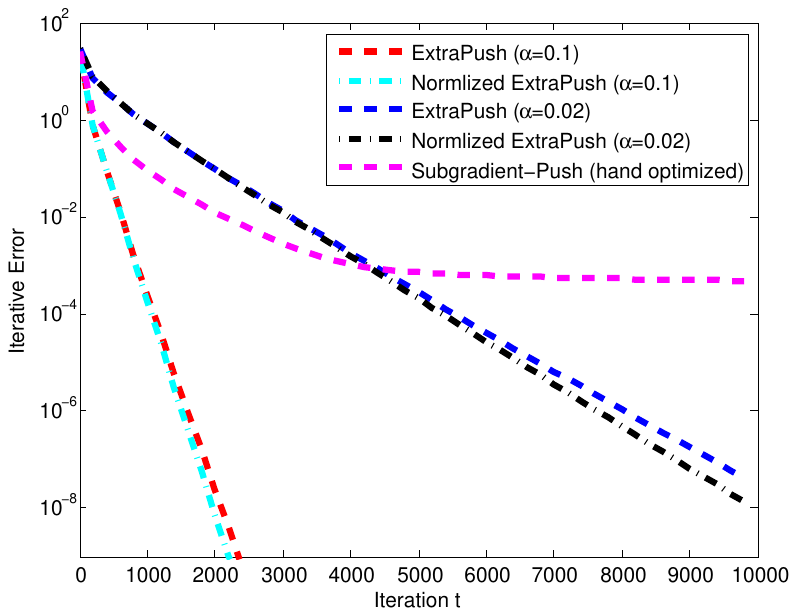}\\
  \caption{Experiment results for decentralized least squares regression. History of $\|{\bf x}^t-{\bf x}^{*}\|_2$, where ${\bf x}^*$ is the exact solution. The performances of ExtraPush and Normalized ExtraPush are very similar.}\label{Fig:DLS}\vskip -4mm
\end{figure}\vskip 2mm


As illustrated in Fig. {\ref{Fig:DLS}}, the performances of ExtraPush and Normalized ExtraPush are almost identical. Their linear convergence rates are affected by different step sizes; a smaller $\alpha$ leads to a slower rate, as one would expect.

\subsection{Decentralized Huber-like Regression}

Instead of least squares, this experiment minimizes the Huber loss function, which is known to be robust to outliers:
\begin{equation}
\label{Eq:DistributedHuber}
x^* \gets \mathop{\mathrm{argmin}}_{x\in \mathbf{R}^p} f(x) = \sum_{i=1}^n f_i(x),
\end{equation}
where $f_i(x) = \sum_{j=1}^{m_i}H_{\xi}(B_{(i)j}x-b_{(i)j}),$ $B_{(i)j}$ is $j$th row of matrix $B_{(i)} \in \mathbf{R}^{m_i\times p}$ and $b_{(i)j}$ is the $j$th entry of vector $b_{(i)} \in \mathbf{R}^{m_i}$  for $i=1, \ldots,n.$ The Huber loss function is defined as
\begin{equation}
\label{Eq:HuberLoss}
H_{\xi}(a) =
\left\{
\begin{array}{lll}
\frac{1}{2}a^2, & \text{for}\ |a| \leq \xi & (\ell_2^2 \ \text{zone}),\\
\xi(|a| - \frac{1}{2} \xi), & \text{otherwise} & (\ell_1 \ \text{zone}).
\end{array}%
\right.
\end{equation}%
Similar to the experimental setting in {\cite{Yin-EXTRA2015}}, we let $\xi  =2$ and set the solution $x^*$ in the $\ell_2^2$ zone while initializing $x_{(i)}^0$ in the $\ell_1$ zone for all agents $i$. Similar to the last experiment, we  teset two different step sizes,  $\alpha = 0.1$ and $0.02$, where $\alpha = 0.1$ is hand-optimized. The step size of the subgradient-push algorithm is hand optimized to $\alpha_t = \frac{5}{\sqrt{t+100}}$. The numerical results are depicted in Fig. {\ref{Fig:DHuber}}.

As shown by Fig. {\ref{Fig:DHuber}}, when $\alpha = 0.1$, both ExtraPush and Normalized ExtraPush algorithms have the sublinear convergence in their first 500 iterations and then show linear convergence, as $x_{(i)}^t$ for most $i$ have entered the $\ell_2^2$ zone. While for $\alpha = 0.02,$ more iterations (about 2500) are needed before both algorithms start decaying linearly.

\begin{figure}[ht]
\centering
  \includegraphics[scale=0.6]{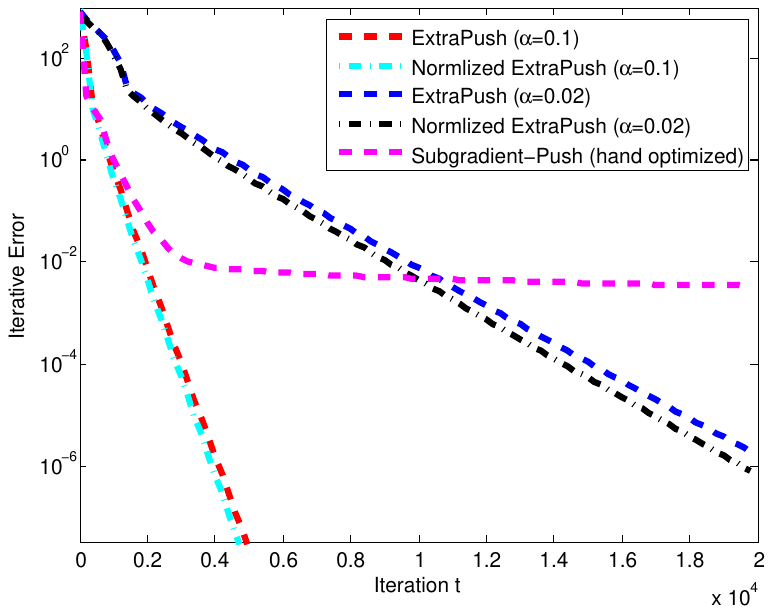}\\
  \caption{Experiment results for decentralized Huber regression. History of $\|{\bf x}^t-{\bf x}^{*}\|_2$, where ${\bf x}^*$ is the exact solution. The performances of ExtraPush and Normalized ExtraPush are very similar.}\label{Fig:DHuber}\vskip -4mm
\end{figure}\vskip 2mm


\section{Conclusion}

In this note, we propose a decentralized algorithm called ExtraPush,
as well as its simplified version called Normalized ExtraPush, for solving  distributed consensus optimization problems over  directed graphs.
The algorithms use  column-stochastic mixing matrices.
We show that Normalized ExtraPush converges at a linear rate if the objective function is smooth and strongly convex.
In additional, we develop the first-order optimality conditions
and provide the convergence of ExtraPush under the boundedness assumption.
The convergence as well as the rate of convergence of ExtraPush should be justified in the future.
Moreover, when applied to a directed time-varying network, the performance of the proposed algorithms will also be studied. Another line of future research is to generalize
ExtraPush to handle the sum of smooth and proximable (possibly nonsmooth) functions as done in \cite{Shi-PGEXTRA2015} that has generalized Extra this way.

\vskip 2mm
\noindent {\bf Acknowledgments.}
The work of W. Yin has been supported in part by the NSF grants DMS-1317602 and ECCS-1462398. The work of J. Zeng has been supported in part by the NSF grant 11501440.

\end{document}